\documentclass{article}
\usepackage[utf8]{inputenc}
\usepackage{amsfonts}
\usepackage{amssymb}
\usepackage{amsmath}
\usepackage{amsthm}
\usepackage{commath}
\usepackage{xcolor}
\usepackage{bbold}
\usepackage{mathrsfs}
\usepackage{enumitem}
\usepackage{graphicx}
\usepackage{cancel}
\usepackage[noabbrev]{cleveref}

\newcommand{\R}{\mathbb{R}}
\newcommand{\T}{\mathbb{T}}
\newcommand{\Z}{\mathbb{Z}}
\newcommand{\C}{\mathbb{C}}

\newcommand{\Id}{\textup{Id}}
\newcommand{\A}{\mathcal{A}}

\newcommand{\dV}{\,d\mathcal{V}}
\newcommand{\wt}[1]{\widetilde{#1}}

\newcommand{\grad}{\textup{grad}\,}
\renewcommand{\del}{\partial}

\newcommand{\M}{\mathcal{M}}
\newcommand{\N}{\mathbb{N}}

\newcommand{\blue}[1]{\textcolor{blue}{#1}}
\renewcommand{\L}{\mathcal{L}}
\newcommand{\loc}{\textit{loc}}
\newcommand{\F}{\mathcal{F}}
\newcommand{\delslash}{\cancel{\del}}
\let\pp\S
\renewcommand{\S}{\mathcal{S}}

\newtheorem{theorem}{Theorem}

\newtheorem{lemma}[theorem]{Lemma}
\newtheorem{corollary}[theorem]{Corollary}
\newtheorem{remark}[theorem]{Remark}

\newtheorem{example}[theorem]{Example}

\numberwithin{theorem}{section}
\title{An approach to Hamiltonian Floer theory for maps from surfaces}
\author{ Ronen Brilleslijper and Oliver Fabert }
\date{\today}

\begin{document}

\maketitle

\begin{abstract}
    In $n$-dimensional classical field theory one studies maps from $n$-dimensional manifolds in such a way that classical mechanics is recovered for $n=1$. In previous papers we have shown that the standard polysymplectic framework in which field theory is described, is not suitable for variational techniques. In this paper, we introduce for $n=2$ a Lagrange-Hamilton formalism that allows us to define a generalization of Hamiltonian Floer theory. As an application, we prove a cuplength estimate for our Hamiltonian equations that yields a lower bound on the number of solutions to Laplace equations with nonlinearity. We also discuss the relation with holomorphic Floer theory.
\end{abstract}

\tableofcontents

\section{Introduction}

This articles continues in the lines of \cite{polysymplFloer,hyperkahlerRookworst}, developing a Hamiltonian formalism for (respectively Minkowski and Euclidean) field theories that is fit for variational calculus. In the current article, the focus lies on the Laplace equation with non-linearity on the two-dimensional torus. The aim of this article is twofold. First of all, we will provide a connection between the Hamilton equations and their Lagrangian counterpart. In this process a nice connection will emerge with complex structures and we will highlight a relation to holomorphic Hamiltonian systems. Secondly, we will prove a cuplength result similar to \cite{hyperkahlerRookworst} in the two-dimensional case using Floer theory. The novel contributions of this paper are the Fredholm theory and compactness results for the moduli space of Floer curves (see \cref{sec:Fredholm,sec:Compactness}). Working out these aspects puts us a step closer to the development of a full Floer theory. Furthermore the analysis should be useful also for defining a holomorphic Floer theory.

The Hamiltonian formalism used in this article arises from what we call \emph{Bridges regularization}. This formalism can be found in \cite{bridgesTEA} and provides a Hamiltonian framework for, amongst others, the Laplace equation that pertains its nice analytical properties. To motivate its merits, let us start by looking at Newton's equation
\begin{align}\label{eq:Newton}
    -\frac{d^2}{dt^2}q(t)=V'_t(q(t)),
\end{align}
for a particle $q:\R\to Q$ subject to some potential $V_t:Q\to\R$ on a manifold $Q$. To formulate this as a first-order equation, we can introduce the momentum $p(t)=\frac{d}{dt}q(t)$ and the Hamiltonian $H_t(u)=\frac{1}{2}|p|^2+V_t(q)$, for $u=(q,p)$ and write \cref{eq:Newton} as
\begin{align}\label{eq:NewtonHam}
    J\frac{d}{dt}u(t)=\nabla H_t(u(t)).
\end{align}
Here $J=\begin{pmatrix}0&-\Id\\\Id&0\end{pmatrix}$ is the standard complex structure on the cotangent bundle $T^*Q$. Note that this so-called Hamiltonian equation is elliptic and that $(J\frac{d}{dt})^2=-\frac{d^2}{dt^2}$. In Floer theory, one uses this Hamiltonian formalism to rewrite solutions to \cref{eq:Newton} as critical points of some action functional $\A_H$. One then studies gradient-lines of $\A_H$ to conclude existence results for solutions to Newton's equation. A well-known result in Floer theory is that the number of critical points of $\A_H$ is bounded below by the topology of $Q$. See \cite{albers2016cuplength} for more explanation and \cite{cieliebak1994pseudo} for the specific case of cotangent bundles. Crucial in the proofs of these results is the ellipticity of \cref{eq:NewtonHam}.

Note that \cref{eq:NewtonHam} comes from a Lagrangian formalism. Periodic solutions to \cref{eq:Newton} arise as critical points of the Lagrangian action functional
\[
\S[q]=\int_{S^1}L(q(t),q'(t),t)\,dt,
\]   
where $L(q,q',t)=\frac{1}{2}(q')^2-V_t(q)$. When we define $p=\frac{\del L}{\del q'}$ and $H=pq'-L$, the Euler-Lagrange equations for the Lagrangian action coincide with \cref{eq:NewtonHam}. Here, in order to view $H_t$ as a function of $q$ and $p$, we need the Legendre condition 
\[
\det\frac{\del^2 L}{\del q'^2}\neq 0. 
\]
We will see a similar condition when connecting the Lagrangian and Hamiltonian approaches for the Laplace equation (see \cref{sec:Lagrangian}). 

\subsection{Two-dimensional Bridges regularization}
In this article, the starting point is the Laplace equation on the 2-dimensional torus. Let $q:\T^2\to\T^d$ satisfy the Laplace equation with non-linearity
\begin{align}\label{eq:Laplace}
    -(\del_1^2+\del_2^2)q(t) = \frac{dV}{dq}(q(t)),
\end{align}
where $t=(t_1,t_2)\in\T^2$, $\del_i:=\del_{t_i}$ and $V:\T^d\to\R$ is a non-linearity. We will prove a cuplength result for the number of solutions to \cref{eq:Laplace} and actually for more general equations. In comparison with Newton's equation we will first try to formulate \ref{eq:Laplace} as a first-order equation. The standard way of doing this, is by introducing 'momentum-like' variables for the different derivatives. The following set of equations follows this so-called de Donder-Weyl approach:\footnote{Note that the de Donder-Weyl equations are typically used as a first-order approach to the wave equation, so the signs are different. See a.o. \cite{kanatchikov1993canonical,gunther1987polysymplectic}.}
\begin{alignat}{3}
    &\del_1 q & &= p_1\nonumber\\
    & &\del_2 q &= p_2\label{eq:DW}\\
    -&\del_1 p_1 &- \del_2 p_2 &= \frac{dV}{dq}\nonumber.
\end{alignat}
Even though \cref{eq:DW} can be written using a Hamiltonian $H$ as $K_\del u(t) = \nabla H(u(t))$ with $u=(q,p_1,p_2)$ for a first-order differential operator $K_\del$, this equation on its own has some major disadvantages. Most notably, the equation is not elliptic anymore, whereas \cref{eq:Laplace} is. For Floer theoretic purposes this is highly undesirable. Moreover, $K_\del$ has an infinite-dimensional kernel on the space of maps on $\T^2$ and it does not square to the Laplacian in contrast to the mechanical case, where $(J\frac{d}{dt})^2=-\frac{d^2}{dt^2}$. As proven in \cite{polysymplFloer} for the case of the wave equation, the Floer curves corresponding to \cref{eq:DW} do not necessarily converge to solutions. Also see \cite{hyperkahlerRookworst} for a similar problem for the Laplace equation on a 3-dimensional torus. 

Fortunately, there is a different way of writing \cref{eq:Laplace} as a first-order system that solves these problems. First proposed in \cite{bridgesderks} and later worked out in \cite{bridgesTEA}, this new set of equations follows from adding some extra terms to \cref{eq:DW}, depicted in blue.
\begin{align}\begin{split}\label{eq:Bridges}
    \del_1 q_1 \blue{+\del_2 q_2}&=p_1\\
    \blue{-\del_1 q_2 +} \del_2 q_1 &= p_2\\
    -\del_1 p_1 - \del_2 p_2 &= \frac{dV_1}{dq_1}\\
    \blue{\del_1p_2-\del_2p_1}&\blue{=\frac{dV_2}{dq_2}}
\end{split}\end{align}
Note that we changed the variable $q$ from \cref{eq:DW} into $q_1$ here, to stress that this first-order system of equations actually describes two Laplace equations 
\begin{align*}
    -(\del_1^2+\del_2^2)q_1(t) &= \frac{dV_1}{dq_1}(q_1(t))\\
    -(\del_1^2+\del_2^2)q_2(t) &= \frac{dV_2}{dq_2}(q_2(t)).
\end{align*}
The underlying geometric structure behind \cref{eq:Bridges}, which we call \textit{Bridges' equations} is explained in \cite{bridgesTEA}. A shorter version in the context of maps between tori can also be found in \cite{polysymplFloer,hyperkahlerRookworst}, so we do not go into details on this in the present article. However, we do want to stress again that \cref{eq:Bridges} is an elliptic first-order system. We can write it as 
\begin{align*}
    J_\del Z(t) = \nabla H(Z(t)),
\end{align*}
where $Z=(q_1,q_2,p_1,p_2)$, $H(Z)=\frac{1}{2}p_1^2+\frac{1}{2}p_2^2+V_1(q_1)+V_2(q_2)$ and $J_\del$ is a first-order operator satisfying $J_\del^2=-(\del_1^2+\del_2^2)$. Note that only constant functions lie in the kernel of $J_\del$. 

\begin{remark}\label{rem:hyperkahler}
    Note that there is a hyperk\"ahler structure on the target space hidden in these equations. When we write out $J_\del=J_1\del_1+J_2\del_2$,then $J_1$ and $J_2$ are both complex structures and their product is too. In \cref{rem:Fueter} we will note that this yields a relation between our Floer curves and the Fueter equation. The complex structure $J_1$ pairs $q_i$ with $p_i$ whereas $J_2$ pairs $q_i$ with $p_j$ for $i\neq j$. Their product $J_3=J_1J_2$ therefore pairs $q_1$ with $q_2$ and $p_1$ with $p_2$. In \cref{sec:complexandlagrange} we will see that indeed \cref{eq:Bridges} can be written using some complex multiplication on $\T^{2d}$.
\end{remark}

\Cref{eq:Bridges} can also be described in terms of the polysymplectic form 
\begin{align}\label{eq:poly}
    \Omega=\omega_1\otimes \del_1+\omega_2\otimes \del_2,
\end{align}
where $\omega_1 = dp_1\wedge dq_1 - dp_2\wedge dq_2$ and $\omega_2 = dp_2\wedge dq_1 + dp_1\wedge dq_2$. This description will be of use later, when we compare the equations to holomorphic Hamiltonian systems. Note that both $\omega_1$ and $\omega_2$ are symplectic forms. Heuristically, $\omega_i$ determines the derivative in the direction of $t_i$. In other words, $Z=(q_1,q_2,p_1,p_2)$ is a solution to \cref{eq:Bridges} if and only if 
\begin{align*}
    \omega_1(\cdot,\del_1 Z)+\omega_2(\cdot,\del_2 Z) = dH.
\end{align*}
See \cite{gunther1987polysymplectic} for an exposition of polysymplectic geometry. 

\subsection{Outline of the article}
This paper is organized as follows. In \cref{sec:complexandlagrange} we start by rewriting the Bridges equations using language from complex analysis. After relating our equations to holomorphic Hamiltonian systems, we discuss the underlying Lagrangian formalism. In \cref{sec:cuplength} we state the main \cref{thm:cuplength} together with its implication for Laplace equatinos with nonlinearity. The $C^0$-bounds are given for both the orbits and the Floer curves. The section concludes with a proof of the theorem, modulo the necessary Fredholm theory and compactness results. These are treated in respectively sections \ref{sec:Fredholm} and \ref{sec:Compactness}. The proof of one technical lemma needed for compactness is deferred to the appendix. 

\section{Complex structures and the Lagrangian formalism}\label{sec:complexandlagrange}
Before motivating the Hamiltonian approach to Laplace's equation mentioned above from the corresponding Lagrangian system, we first want to note that we can rewrite \cref{eq:Bridges} in terms of complex structures. Note that for $H\equiv 0$ (i.e. a vanishing right-hand side in \cref{eq:Bridges}), Bridges' equations reduce to a set of two Cauchy Riemann equations for the pairs $(q_1,q_2)$ and $(p_1,p_2)$. This leads us to define $t=t_1+it_2$ and correspondingly $\bar{t}=t_1-it_2$. Also let $q=q_1+iq_2$, $p=p_1+ip_2$ and correspondingly $\bar{q}=q_1-iq_2$, $\bar{p}=p_1-ip_2$. Then \cref{eq:Bridges} is equivalent to the set of equations
\begin{align}\begin{split}\label{eq:complex}
    \del_t q &= \frac{1}{2}\bar{p} = \frac{\del H}{\del p}\\
    -\del_t p &= \frac{\del V}{\del q} = \frac{\del H}{\del q},
\end{split}\end{align}
where $\del_t = \frac{1}{2}(\del_1-i\del_2)$, $\frac{\del}{\del q} = \frac{1}{2}(\frac{\del}{\del q_1} -i\frac{\del}{\del q_2})$ and $V(q,\bar{q})=V_1(q_1)+V_2(q_2)\in\R$. Also $H=H(q,\bar{q}, p, \bar{p})=\frac{1}{2}p\bar{p}+V(q,\bar{q})$ is still the same real-valued function as before, but now seen as a function of these new coordinates.  

While the Hamiltonian formalism for the Laplace equation might come across as purely notation at this point, we will show in the \cref{sec:Lagrangian} that these equations actually relate to a Lagrangian version of a minimization problem by means of a Legendre transform. This will also provide us with a much bigger class of examples then just Laplace equations. Before that, the next section explains the relation to holomorphic Hamiltonian systems.

\subsection{Relation to holomorphic Hamiltonian systems}\label{rem:HHS}
    \Cref{eq:complex} looks very similar to the holomorphic Hamiltonian systems (HHS) studied in \cite{wagner2023pseudo}. The main difference is that their Hamiltonians are holomorphic functions with values in $\C$, whereas our Hamiltonian is real-valued. We crucially need the derivative $\frac{\del H}{\del p}$ to actually depend on $\bar{p}$ in order to get the Laplace equation
    \begin{align}\label{eq:complexLaplace}-\Delta (q_1+iq_2) = -4\del_{\bar{t}}\del_t q = -2\del_{\bar{t}}\bar{p}=2\frac{\del V}{\del \bar{q}} = \frac{d V_1}{d q_1}+i\frac{d V_2}{d q_2}.\end{align}
    Notice also that our Hamiltonians are not necessarily the real or imaginary parts of a holomorphic Hamiltonian, since we do not require them to be harmonic.

    Interestingly, both the HHS's from Wagner and the Bridges equations used in this article can be described as special cases of a broader framework. Recall the polysymplectic form $\Omega$ from \cref{eq:poly}. In the complex notation it can be written as $\Omega=\omega_\C\otimes\del_t + \bar{\omega}_\C\otimes\del_{\bar{t}}$, where $\omega_\C = dp\wedge dq$ and $\bar{\omega}_\C=d\bar{p}\wedge d\bar{q}$. Just as in \cite{wagner2023pseudo}, the form $\omega_\C$ (respectively $\bar{\omega}_\C$) induces an isomorphisms between the holomorphic (respectively anti-holomorphic) tangent and cotangent bundles
    \begin{align*}
        \omega_\C^\flat&:T^{(1,0)}X\overset{\sim}{\to} T^{*(1,0)}X\\
        \bar{\omega}_\C^\flat&:T^{(0,1)}X\overset{\sim}{\to} T^{*(0,1)}X,
    \end{align*}
    where $X=T^*\T^{2d}$. Thus, the polysymplectic form $\Omega$ induces an isomorphism $\omega_\C^\flat\oplus\bar{\omega}_\C^\flat:T_\C X\overset{\sim}{\to} T^*_\C X$, where the complexified tangent and cotangent bundles are viewed as the direct sums of their holomorphic and anti-holomorphic subbundles. In principle this means that for any function $H:X\to\C$ we get a complex vector field by taking the inverse of $dH$ under this map. In the study of HHS's $H$ is taken to be a holomorphic function, which means that $dH\in T^{*(1,0)}X$. Thus the Hamiltonian equations only take into account the holomorphic vector field $(\omega_\C^\flat)^{-1}(dH)$ and the anti-holomorphic vector field vanishes. In our case the Hamiltonians are real-valued meaning that its holomorphic and anti-holomorphic derivatives are equal. Thus $dH$ lies in the diagonal of $T^*_\C X = T^{*(1,0)}X\oplus T^{*(0,1)}X$ and we actually get two sets of conjugate equations. The analysis in this article is aimed at the case of real-valued Hamiltonians, but we expect similar techniques to work in the holomorphic setting.

    \begin{remark}
        Note that \cref{eq:complex} is written in terms of local coordinates on the torus. However, just as for HHS's, the equations can be formulated on more general manifolds. For \cref{eq:complex} to make sense the domain should be a Riemann surface and the target a complex manifold. From a dynamical point of view this could be of particular interest for string theory, where the Riemann surface is viewed as the closed world-sheet of a 1-dimensional string and the target manifold is Calabi-Yau. The flatness of the Calabi-Yau manifold will be crucial for the compactness result to hold, just like in \cite{hohloch2009hypercontact}.
    \end{remark}

\subsection{Lagrangian approach}\label{sec:Lagrangian}
We observe that solutions to the Laplace equations arise as critical points of the functional
\begin{align*}
    \S[q]=-\int_{\T^2} \left(2\del_t q \del_{\bar{t}}\bar{q}-V(q,\bar{q})\right)\frac{dt\wedge d\bar{t}}{2i},
\end{align*}
where we use the notation of \cref{eq:complexLaplace} (for a proof see \cref{lem:EL} below). In this section, we will develop a connection between the Lagrangian picture and the Hamiltonian formalism sketched above.

Let $L=L(q,\bar{q},\del_t q, \del_{\bar{t}}\bar{q},t,\bar{t})$ be a Lagrangian with values in $\R$ and 
\begin{align}\label{eq:action}\S[q]=-\int_{\T^2}L(q,\bar{q},\del_t q, \del_{\bar{t}}\bar{q},t,\bar{t})\frac{dt\wedge d\bar{t}}{2i}\end{align}
the corresponding action. Note that $\S$ is real-valued as well. Our main example is the Lagrangian $L=2\del_tq\del_{\bar{t}}\bar{q}-V(q,\bar{q})$ with action given above. 
\begin{lemma}\label{lem:EL}
    The Euler-Lagrange equations for the action (\ref{eq:action}) are given by
    \begin{align}\label{eq:EL}
        \frac{\del L}{\del q} = \del_t \frac{\del L}{\del(\del_t q)} && \frac{\del L}{\del \bar{q}} = \del_{\bar{t}} \frac{\del L}{\del(\del_{\bar{t}} \bar{q})}.
    \end{align}
    Note that these two equations are each others conjugates as $L$ is real. 
\end{lemma}
\begin{proof}
    Let us calculate the critical points of $\S$. To this extent we take a family of maps $q_s$ where $s\in\R$ such that $q_0=q$ and write $\frac{dq_s}{ds}\vert_{s=0}=\dot{q}$. Then
    \begin{align*}
        \left.\frac{d}{ds}\right\vert_{s=0} \S[q_s] &= -\int_{\T^2}\left( \frac{\del L}{\del q}\dot{q}+\frac{\del L}{\del \bar{q}}\dot{\bar{q}} +\frac{\del L}{\del(\del_t q)}\del_t\dot{q} +\frac{\del L}{\del(\del_{\bar{t}}\bar{q})}\del_{\bar{t}}\dot{\bar{q}} \right)\frac{dt\wedge d\bar{t}}{2i}\\
        &=-\int_{\T^2} \left(\dot{q}\left(\frac{\del L}{\del q}-\del_t\frac{\del L}{\del(\del_t q)}\right)+\dot{\bar{q}}\left(\frac{\del L}{\del \bar{q}}-\del_{\bar{t}}\frac{\del L}{\del(\del_{\bar{t}} \bar{q})}\right)\right)\frac{dt\wedge d\bar{t}}{2i}.
    \end{align*}
    We see that this derivative vanishes precisely when \cref{eq:EL} is satisfied.
\end{proof}

In order to get back to a Hamiltonian formalism we define a momentum-like variable $p=\frac{\del L}{\del(\del_t q)}$ and correspondingly $\bar{p}=\frac{\del L}{\del(\del_{\bar{t}}\bar{q})}$. This yields 
\begin{align}\label{eq:dL}
    dL &= d(p\del_t q + \bar{p}\del_{\bar{t}}\bar{q})-\del_t qdp-\del_{\bar{t}}\bar{q}d\bar{p}+\frac{\del L}{\del q}dq+\frac{\del L}{\del \bar{q}}d\bar{q}+\frac{\del L}{\del t}dt+\frac{\del L}{\del \bar{t}}d\bar{t},
\end{align} 
which motivates the definition \[H(q,\bar{q},p,\bar{p},t,\bar{t})=p\del_t q +\bar{p}\del_{\bar{t}}\bar{q}-L(q,\bar{q},\del_t q, \del_{\bar{t}}\bar{q},t,\bar{t}),\] where $(\del_t q,\del_{\bar{t}}\bar{q})$ is seen as a function of $(q,\bar{q},p,\bar{p},t,\bar{t})$ by the inverse function theorem. This requires the Legendre-like condition \[\det\left(\left(\frac{\del^2 L}{\del(\del_t q)\del(\del_{\bar{t}}\bar{q})}\right)^2-\frac{\del^2 L}{\del(\del_t q)^2}\frac{\del^2 L}{\del(\del_{\bar{t}}\bar{q})^2}\right)\neq 0.\] Note again that $H$ is real-valued. 

\begin{lemma}
    The Euler-Lagrange equations (\ref{eq:EL}) are equivalent to the Hamiltonian equations
    \begin{align}\begin{split}\label{eq:Ham}
        \del_t q &= \frac{\del H}{\del p}\\
        -\del_t p &= \frac{\del H}{\del q}.
    \end{split}\end{align}
\end{lemma}
\begin{proof}
    By the relation (\ref{eq:dL}) above we get that
    \[dL = d(p\del_t q + \bar{p}\del_{\bar{t}}\bar{q}) - dH.\]
    So
    \begin{align*}
        \frac{\del H}{\del q} = -\frac{\del L}{\del q} && \frac{\del H}{\del p} = \del_t q. 
    \end{align*}
    Plugging in the Euler-Lagrange equations and the definition of $p$, the result follows immediately. 
\end{proof}
\begin{example}
    In the example $L=2\del_tq\del_{\bar{t}}\bar{q}-V(q,\bar{q})$ of the Laplace equations, we get $H=\frac{1}{2}p\bar{p}+V(q,\bar{q})$ as before and we recover \cref{eq:complex}.
\end{example}

\section{Cuplength proof}\label{sec:cuplength}
In this section we will prove a cuplength result for solutions to the Hamiltonian equations (\ref{eq:Ham}). Recall that we view $t=t_1+it_2$ and its conjugate as the coordinates on $\T^2$ and $q=q_1+iq_2$ and its conjugate as the coordinates on $\T^{2d}$. Consequently, $Z:=(q,\bar{q},p,\bar{p})$ denote the coordinates on $T^*\T^{2d}$, where $p=p_1+ip_2$. 
\begin{theorem}\label{thm:cuplength}
    Let $H:\T^2\times T^*\T^{2d}\to \R$ be given by $H_{t,\bar{t}}(Z)=\frac{1}{2}p\bar{p}+F_{t,\bar{t}}(Z)$, where $F$ is smooth, real-valued and has finite $C^3$-norm. Then there exist at least $(2d+1)$ solutions to \cref{eq:Ham}.
\end{theorem}
\begin{corollary}
    The set of Laplace equations (\ref{eq:complexLaplace}) has at least $(2d+1)$ solutions, when $V_1$ and $V_2$ have finite $C^3$-norm. 
\end{corollary}

\Cref{thm:cuplength} will be proven using a Floer theoretic argument along the lines of \cite{schwarz1998quantum,albers2016cuplength}. Note that similar statements have been proven in \cite{hohloch2009hypercontact,ginzburg2012hyperkahler,ginzburg2013arnold,albers2016cuplength} for closed target manifolds and the proof for maps on a 3-dimensional torus was given in \cite{hyperkahlerRookworst}. In the latter article, the proof was given as a corollary of results from \cite{ginzburg2012hyperkahler} after establishing the necessary $C^0$-bounds. This proof does not use Floer theoretic arguments. In the present article we aim to establish a connection to the Floer theory of these systems and will prove the necessary Fredholm and compactness results (see \cref{sec:Fredholm,sec:Compactness} respectively). As pointed out in \cref{rem:HHS}, the techniques used here should also be of use for studying holomorphic Hamiltonian systems in the sense of \cite{wagner2023pseudo} using Floer theory.

\subsection{$C^0$-bounds}
First, we will prove the necessary $C^0$-bounds. The proofs are analogous to the ones given in \cite{hyperkahlerRookworst}, but as the notation with complex structures is different in the present article and we treat the 2-dimensional torus here, the proofs will be briefly be repeated below. First of all, define the action by
\begin{align}\label{eq:hamaction}\A_H(Z)=\int_{\T^2} \left(p\del_t q+\bar{p}\del_{\bar{t}}\bar{q}-H_{t,\bar{t}}(Z)\right)\dV,\end{align}
where $\dV=-\frac{dt\wedge d\bar{t}}{2i}$ is the standard volume form on $\T^2$. In the case that $H$ actually comes from a Lagrangian as in \cref{sec:Lagrangian}, this action coincides with \cref{eq:action}. We denote 
\begin{align*}
    \mathcal{P}(H)&=\{Z:\T^2\to T^*\T^{2d}\mid Z\textup{ is nullhomotopic and satisfies \cref{eq:Ham}}\}\\
    \mathcal{P}_a(H)&=\{Z\in\mathcal{P}(H)\mid a\geq \A_H(Z)\}.
\end{align*}

\begin{lemma}\label{lem:Linfty}
    Let $H_{t,\bar{t}}(Z)=\frac{1}{2}p\bar{p}+F_{t,\bar{t}}(Z)$, where $F$ is smooth, real-valued and has finite $C^2$-norm. Then for any $a\in\R$ the set $\mathcal{P}_a(H)$ is bounded in $L^\infty$. 
\end{lemma}
\begin{proof}
    We start by proving that $\mathcal{P}_a(H)$ is bounded in $L^2$. Note that for solutions of \cref{eq:Ham} we get
    \begin{align*}
        \A_H(Z) &= \int_{\T^2} \left( p\frac{\del H}{\del p} +\bar{p}\frac{\del H}{\del \bar{p}} -H_{t,\bar{t}}(Z)\right)\dV\\
        &=\int_{\T^2}\left( \frac{1}{2}p\bar{p} +p\frac{\del F}{\del p}+\bar{p}\frac{\del F}{\del \bar{p}}-F_{t,\bar{t}}\right)\dV\\
        &\geq h_0||p||_{L^2}^2-h_1,
    \end{align*}
    for some $h_0>0$ and $h_1\geq 0$. The last inequality comes from the boundedness of $F$ and its derivatives. The assumption $a\geq\A_H(Z)$ now yields a uniform $L^2$-bound on $p$. As $q$ maps into $\T^{2d}$ it is $L^2$-bounded by construction. So indeed $\mathcal{P}_a(H)$ is bounded in $L^2$. 

    We now proceed by a bootstrapping argument. For $Z\in\mathcal{P}_a(H)$ we have\footnote{In the first line we use partial integration, using that $Z$ is nullhomotopic.}
    \begin{align*}
        ||dZ||^2_{L^2} &= 4\int_{\T^2}\left(|\del_t q|^2+|\del_t p|^2\right)\dV\\
        &=4\int_{\T^2}\left(\left|\frac{\del H}{\del p}\right|^2+\left|\frac{\del H}{\del q}\right|^2\right)\dV\\
        &\leq ||p||_{L^2}^2 + 4\int_{\T^2}\left(\left|\frac{\del F}{\del q}\right|^2+\left|\frac{\del F}{\del p}\right|^2\right)\dV.
    \end{align*}
    As the first derivatives of $F$ are bounded and $p$ is uniformly bounded in $L^2$ this induces a uniform $L^2$-bound on $dZ$. Therefore $\mathcal{P}_a(H)$ is bounded in $H^1$. A similar argument shows that the $H^1$-norm of $dZ$ is bounded by the $H^1$-norm of $Z$ and the $C^2$-norm of $F$. Thus, we get that $\mathcal{P}_a(H)$ is bounded in $H^2$. As $2\cdot 2>2$ we get boundedness in $L^\infty$ by Sobolev's embedding theorem. 
\end{proof}

From \cref{lem:Linfty} it follows that we can actually use a cut-off version of $F$. Let $\chi_\rho:[0,\infty)\to\R$ a smooth cut-off function such that $\chi_\rho\equiv 1$ on $[0,\rho-1]$ and $\chi_\rho\equiv 0$ on $[\rho,\infty)$. We can define the cut-off Hamiltonian as $\tilde{H}^\rho_{t,\bar{t}}(Z):=\frac{1}{2}p\bar{p}+\tilde{F}^\rho_{t,\bar{t}}(Z):=\frac{1}{2}p\bar{p}+\chi_\rho(|p|^2)F_{t,\bar{t}}(Z)$.
\begin{corollary}\label{cor:cutoffHam}
    For any $a\in\R$ there exists some $\rho>0$ such that $\mathcal{P}_a(H)=\mathcal{P}_a(\tilde{H}^\rho)$.
\end{corollary}

\subsection{Floer curves}
Now that we have the necessary bounds on the orbits, we turn to Floer curves. A Floer curve is defined as a negative $L^2$-gradient flow line of the action (\ref{eq:hamaction}). Note that in the complex notation the $L^2$-inner product is given by
\[\langle X,Y\rangle_{L^2} = \frac{1}{2}\int_{T^2}\left(X_q Y_{\bar{q}}+X_{\bar{q}}Y_q+X_pY_{\bar{p}}+X_{\bar{p}}Y_p\right)\dV.\]
It follows that the Floer curves are given by
\begin{alignat}{4}
    \frac{1}{2}\del_s q-\del_{\bar{t}}\bar{p} &= \frac{\del \tilde{H}^\rho}{\del \bar{q}} && \qquad && \frac{1}{2}\del_s \bar{q}-\del_t p &= \frac{\del \tilde{H}^\rho}{\del q}\nonumber\\
    \frac{1}{2}\del_s p+\del_{\bar{t}}\bar{q} &= \frac{\del \tilde{H}^\rho}{\del \bar{p}} && \qquad && \frac{1}{2}\del_s \bar{p}+\del_t q &= \frac{\del \tilde{H}^\rho}{\del p},\label{eq:Floer}
\end{alignat}
where $Z=(q,\bar{q},p,\bar{p}):\R\times\T^2\to T^*\T^{2d}$ and we have taken the Hamiltonian to be $\tilde{H}^\rho$ in view of \cref{cor:cutoffHam}. Note that $s\in\R$ is a real coordinate, so that $\del_s$ denotes the ordinary derivative of functions on $\R$, whereas $\del_t$ and $\del_{\bar{t}}$ still denote the complex derivative and its conjugate on $\T^2$.

\begin{remark}\label{rem:Fueter}
    Note that in the notation of \cref{rem:hyperkahler} we can write the Floer equation with vanishing Hamiltonian as $\del_s Z +J_1\del_1 Z+J_2\del_2 Z=0$. This is equivalent to the Fueter equation
    \[J_3\del_s Z + J_2\del_1 Z - J_1\del_2 Z = 0,\]
    where $J_3=J_1J_2$ is the third complex structure from the hyperk\"ahler structure on the target space.
\end{remark}

The following lemma shows that the image of a Floer curve actually sits within a compact subset of $T^*\T^{2d}$ (compare with Lemma 3.4 from \cite{hyperkahlerRookworst}).
\begin{lemma}\label{lem:Floerbounds}
    Let $Z:\R\times\T^2\to T^*\T^{2d}$ be a solution to \cref{eq:Floer} which converges to critical points of $\A_{\tilde{H}^\rho}$ for $|s|\to\infty$. Then $|p(s,t)|^2\leq \rho$ for any $(s,t)\in\R\times\T^2$.
\end{lemma}
\begin{proof}
    Suppose $|p(s_0,t_0)|^2>\rho$ for some $(s_0,t_0)\in\R\times \T^2$. Locally around $(s_0,t_0)$ the Hamiltonian reduces to $\tilde{H}^\rho(Z(s,t))=\frac{1}{2}p(s,t)\bar{p}(s,t)$. Therefore, on this neighbourhood
    \begin{align*}
        \del_s^2 p &=\del_s p - 2\del_s\del_{\bar{t}}\bar{q}\\
        &= \del_s p - 4\del_t\del_{\bar{t}} p.
    \end{align*}
    In the second equality, we use that $\del_s \bar{q}=2\del_t p$. It follows that 
    \begin{align*}(\del_s^2 +\Delta)p = \del_s^2 p + 4\del_t\del_{\bar{t}}p=\del_s p.\end{align*}

    Define $\phi=\frac{1}{2}p\bar{p}$. Then
    \begin{align*}
        (\del_s^2+\Delta)\phi &= |\del_s p|^2 + 2|\del_t\bar{p}|^2+2|\del_t p|^2 + \frac{1}{2}p\del_s \bar{p} +\frac{1}{2}\bar{p}\del_s p\\
        &\geq \del_s \phi.
    \end{align*}
    The maximum principle now tells us that $\phi$ cannot attain a maximum in the neighbourhood of $(s_0,t_0)$, contradicting the assumption that the Floer curve converges as $|s|\to\infty$. 
\end{proof}

\subsection{Proof of \cref{thm:cuplength}}
Having established the $C^0$-bounds we can prove \cref{thm:cuplength}. The necessary Fredholm and compactness results will be deferred to \cref{sec:Fredholm,sec:Compactness}. For the proof we will follow the lines of the paper \cite{albers2016cuplength}. 

To start the proof, assume there are finitely many solutions to \cref{eq:Ham}, since otherwise we would be done. Then their action is necessarily bounded. \Cref{cor:cutoffHam} implies that we may replace $H$ by $\tilde{H}^\rho$ for some $\rho$ and by \cref{lem:Floerbounds} also the Floer curves will take value in some compact set $\T^{2d}\times B\subseteq T^*\T^{2d}$. Let $M\subseteq C^\infty(\T^2,\T^{2d}\times B)$ be the subset of nullhomotopic maps. Just like in \cite[\pp 3.1]{albers2016cuplength} we define a smooth family of functions $\beta_r:\R\to[0,1]$ for $r\geq 0$ such that 
\begin{itemize}
    \item $\beta_r(s)=0$ for $s\leq -1$ and $s\geq (2d+1)r+1$ for all $r\geq 0$,
    \item $\beta_r|_{[0,(2d+1)r]}\equiv 1$ for $r\geq 1$,
    \item $0\leq \beta_r'(s)\leq 2$ for $s\in(-1,0)$ and \\$0\geq \beta_r'(s)\geq -2$ for $s\in((2d+1)r,(2d+1)r+1)$,
    \item $\lim_{r\to 0^+}\beta_r=0$ in the strong topology on $C^\infty$.
\end{itemize}
See \cite{albers2016cuplength} for the precise conditions on $\beta_r$ and \cref{fig:beta} for an idea of what they look like. We will use the functions $\beta_r$ to interpolate between the quadratic Hamiltonian with and without non-linearity. To this extent define
\[
\tilde{H}^{\rho,r}_{s,t,\bar{t}}(Z)=\frac{1}{2}p\bar{p}+\beta_r(s)\tilde{F}^\rho_{t,\bar{t}}(Z).
\]
One can see the Hamiltonian $\tilde{H}^{\rho,r}$ as the quadratic Hamiltonian $H^0(Z)=\frac{1}{2}p\bar{p}$ with the non-linearity $\tilde{F}^\rho$ switched on just on an interval whose length is determined by $r$. Correspondingly, we have the action $G_{r,s}:=\A_{\tilde{H}^{\rho,r}}$, interpolating between the actions $\A_{H^0}$ and $\A_{\tilde{H}^\rho}$. 
\begin{figure}
    \centering
    \includegraphics[width=0.8\linewidth]{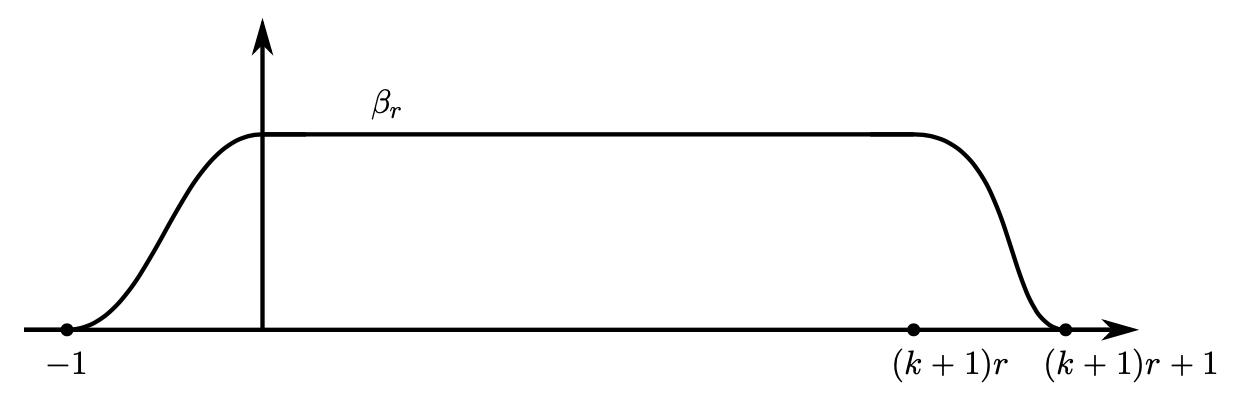}
    \caption{The functions $\beta_r$. Source: \cite{albers2016cuplength}. In our case we take $k=\textup{cl}(\T^{2d})=2d$.}
    \label{fig:beta}
\end{figure}

With this setup we can introduce the moduli space
\begin{align}\label{eq:ModuliSpace}
    \M = \{(r,Z)\mid r\geq 0, Z:\R\to M, (\star 1)-(\star3)\},
\end{align}
where $(\star1)-(\star3)$ are given below. Note that we can view $Z$ as a contractible map $\R\times\T^2\to\T^{2d}\times B$. The condition $(\star1)$ is that $Z$ is a negative gradient line of $G_{r,s}$, meaning a Floer curve for $\tilde{H}^{\rho,r}$. In formula this means
\begin{align}
    \tag{$\star1$} \frac{1}{2}\del_s Z +\delslash Z = \begin{pmatrix}\del_{\bar{q}}\\\del_q\\\del_{\bar{p}}\\\del_p\end{pmatrix}\tilde{H}^{\rho,r}(Z),
\end{align}
where
\[\delslash = \begin{pmatrix}0&0&0&-\del_{\bar{t}}\\0&0&-\del_t&0\\0&\del_{\bar{t}}&0&0\\\del_t&0&0&0\end{pmatrix}\]
and as before $Z=(q,\bar{q},p,\bar{p})$. Condition $(\star2)$ is finiteness of energy
\begin{align}
    \tag{$\star2$} E(Z) = \int_{-\infty}^\infty\int_{\T^2}|\del_s Z(s,t)|^2\dV\,ds <+\infty
\end{align}
and the final condition is convergence of the $p$-coordinates
\begin{align}
    \tag{$\star3$} p(s,t)\to 0\textup{ as }|s|\to\infty.
\end{align}
We also introduce the subspaces
\begin{align*}
    \M_{[0,R]}(\T^{2d})&:=\{(r,Z)\in\M\mid 0\leq r\leq R\textup{ and }\lim_{s\to\pm\infty}Z(s)\in\T^{2d}\times\{0\}\}\\
    \M_R(\T^{2d})&:=\{Z\mid (R,Z)\in\M_{[0,R]}(\T^{2d})\},
\end{align*}
where $\T^{2d}\times\{0\}\subseteq M$ denotes the set of constant maps. This coincides with the set of critical points of $\A_{H^0}$.  Note that for $R=0$ we get that $\M_0(\T^{2d})$ is equal to this set of critical points. 

It follows from \cref{thm:manifold} and \cref{thm:compactness} that $\M$ is a compact 1-dimensional manifold, possibly after a small perturbation of $F$. Also the spaces $\M_{[0,R]}(\T^{2d})$ and $\M_R(\T^{2d})$ will then be compact manifolds (again after possibly perturbing $F$, see \cite{albers2016cuplength}). We have an evaluation map
\begin{align*}
    \textup{ev}_R:\M_R(Z)&\to (\T^{2d}\times B)^{2d}\\
    Z&\mapsto (Z(R,t_0,\bar{t}_0),Z(2R,t_0,\bar{t}_0),\ldots,Z(2dR, t_0,\bar{t}_0)),
\end{align*}
where $(t_0,\bar{t}_0)\in\T^2$ is some basepoint. For $R=0$ this gives the diagonal embedding $\T^{2d}\hookrightarrow (\T^{2d}\times B)^{2d}$.

The rest of the proof goes analogously to \cite{albers2016cuplength}. Take Morse functions $f_1,\ldots,f_{2d},f_*$ on $\T^{2d}$ and extend to Morse functions $\bar{f}_\alpha$ for $\alpha=1,\ldots,2d$ on $\T^{2d}\times B$ using a positive definite quadratic form on a tubular neighbourhood of $\T^{2d}$ in the direction normal to $\T^{2d}$. We also choose Riemannian metrics $g_1,\ldots,g_{2d},g_*$ on $\T^{2d}\times B$. For critical points $x_\alpha\in\T^{2d}$ of $f_\alpha$ and $x_*^\pm\in\T^{2d}$ of $f_*$ we define
\[
\M(R,x_1,\ldots,x_{2d},x_*^-,x_*^+):=\left\{Z\in\M_R(\T^{2d})\middle\vert (\star4)-(\star6)\right\},
\]
where
\begin{align}
    \tag{$\star4$} \textup{ev}_R&\in\prod_{\alpha}W^s(x_\alpha,\bar{f}_\alpha,g_\alpha)\\
    \tag{$\star5$} \lim_{s\to-\infty}Z(s)&\in W^u(x_*^-,f_*,g_*)\\ 
    \tag{$\star6$} \lim_{s\to+\infty}Z(s)&\in W^s(x_*^+,f_*,g_*).
\end{align}
Here, $W^u$ and $W^s$ denote the unstable and stable manifolds of the critical points for the gradient flows of the indicated Morse functions. Again, by choice of perturbations, we may assume these are all smooth manifolds. Note that $\M(0,x_1,\ldots,x_{2d},x_*^-,x_*^+)=\bigcap_{\alpha}W^s(x_\alpha,f_\alpha,g_\alpha)\cap W^u(x_*^-,f_*,g_*)\cap W^s(x_*^+,f_*,g_*)$.

Define the following operation on Morse co-chain groups
\begin{align*}
\theta_R&:\bigotimes_\alpha CM^*(f_\alpha)\otimes CM_*(f_*)\to CM_*(f_*)\\
x_1\otimes\cdots\otimes x_{2d}\otimes x_*^-&\mapsto \sum_{x^+_*\in\textup{Crit}(f_*)}\#_2\M(R,x_1,\ldots,x_{2d},x_*^-,x_*^+)\cdot x^+_*.
\end{align*}
Here, $\#_2$ denotes the parity of the set if it is discrete and 0 if it not. For generic choices of Morse functions and Riemannian metrics, the map $\theta_0$ realizes the usual cup-product in cohomology (see \cite{cupproduct}). This map does not vanish as the cuplength of $\T^{2d}$ is $2d$. Moreover, $\theta_R$ is chain homotopic to $\theta_0$ and can therefore not vanish either. Thus there have to be critical points $x_\alpha$ of $f_\alpha$ and $x_*^\pm$ of $f_*$ of index at least 1 such that the moduli spaces $\M(n,x_1,\ldots,x_{2d},x_*^-,x_*^+)$ are non-empty for all $n\in\N$. 

Let $Z_n\in\M(n,x_1,\ldots,x_{2d},x_*^-,x_*^+)$ and 
\[
Z^{(\alpha)}=\lim_{n\to\infty}Z_n(\cdot+n\alpha),
\]
for $\alpha=1,\ldots,2d$.
By the compactness discussed in \cref{sec:Compactness} these limits exist in $C^\infty_{loc}(\R,M)$. The $Z^{(\alpha)}$ satisfy the Floer \cref{eq:Floer} for the Hamiltonian $\tilde{H}^\rho$
and they converge to critical points $y^\pm_{\alpha}$ of $\A_{\tilde{H}^\rho}$ for $s\to\pm\infty$ such that
\[
\A_{\tilde{H}^\rho}(y_1^-)\geq \A_{\tilde{H}^\rho}(y_1^+)\geq \A_{\tilde{H}^\rho}(y_2^-)\geq \A_{\tilde{H}^\rho}(y_2^+)\geq\ldots\geq \A_{\tilde{H}^\rho}(y_{2d}^-)\geq \A_{\tilde{H}^\rho}(y_{2d}^+).
\] 
The only thing left to do is show that at least $2d+1$ of these points are different, by showing that the inequalities $\A_{\tilde{H}^\rho}(y_\alpha^-)>\A_{\tilde{H}^\rho}(y_\alpha^+)$ are strict. 

Note that since $\textup{Crit}(\A_{\tilde{H}^\rho})$ is finite by assumption, we may assume critical points of $\A_{\tilde{H}^\rho}$ do not intersect the stable manifolds $W^s(x_\alpha,\bar{f}_\alpha,g_\alpha)$. If $\A_{\tilde{H}^\rho}(y_\alpha^-)=\A_{\tilde{H}^\rho}(y_\alpha^+)$, then $Z^{(\alpha)}$ is constant and therefore a critical point of $\A_{\tilde{H}^\rho}$. Also note that by definition $Z_n(n\alpha)\in W^s(x_\alpha,\bar{f}_\alpha,g_\alpha)$ for any $n$. Therefore $Z^{(\alpha)}(0)=\lim_{n\to\infty}Z_n(n\alpha)$ lies in the closure of a stable manifold, which itself is a union of stable manifolds. This is in contradiction with $Z^{(\alpha)}$ being a critical point of $\A_{\tilde{H}^\rho}$. Thus indeed the inequalities $\A_{\tilde{H}^\rho}(y_\alpha^-)>\A_{\tilde{H}^\rho}(y_\alpha^+)$ are strict and we get that $y_1^-,y_2^-,\ldots,y_{2d}^-$ and $y_{2d}^+$ form $2d+1$ different critical points of $\A_{\tilde{H}^\rho}$. These correspond to $2d+1$ distinct solutions of \cref{eq:Ham} \qed


\section{Fredholm theory}\label{sec:Fredholm}
In this section we will prove the following theorem.
\begin{theorem}\label{thm:manifold}
    The moduli space $\M$ from \cref{eq:ModuliSpace} is a 1-dimensional manifold for generic choice of $F$.
\end{theorem}
As is usual in Floer theory, this will be proven by describing $\M$ as the zero set of a Fredholm map. The results are standard, so we will not go into full detail. However, we would like to give the proof of the Fredholmness, to stress that in our setting elementary Fourier arguments can be used. 

We let $\F$ denote the Floer map
\[\F(r,Z)=\F^r(Z)=\frac{1}{2}\del_s Z +\delslash Z - \begin{pmatrix}\del_{\bar{q}}\\\del_q\\\del_{\bar{p}}\\\del_p\end{pmatrix}\tilde{H}^{\rho,r}(Z),\]
where as before $Z=(q,\bar{q},p,\bar{p})$. This means that $\M=\F^{-1}(0)$. We start by proving $\F$ is a nonlinear Fredholm map when acting on the appropriate spaces.

\subsection{Linearization of the Floer map}
Let $r$ be fixed. Then the linearization of $\F^r$ at $Z$ is given by
\begin{align}\label{eq:linFloer}
(d\F^r)_Z = D-\beta_r(s)S_Z,
\end{align}
where $D=\frac{1}{2}\del_s+\delslash-\frac{1}{2}P$,
\begin{align*}
    P &= \begin{pmatrix}0&0&0&0\\0&0&0&0\\0&0&1&0\\0&0&0&1\end{pmatrix} &&\textup{and}&& S_Z=\begin{pmatrix}0&1&0&0\\1&0&0&0\\0&0&0&1\\0&0&1&0\end{pmatrix}\textup{Hess}_Z\tilde{F}^\rho.
\end{align*}
By $H^k$ we denote the Sobolev space $W^{k,2}$. Note that for $k>\frac{3}{2}$ we have an embedding $H^k(\R\times\T^2,\R^{4d})\hookrightarrow C^0(\R\times\T^2,\R^{4d})$.
\begin{lemma}\label{lem:bijection}
    The operator $D:H^k(\R\times\T^2,\R^{4d})\to H^{k-1}(\R\times\T^2,\R^{4d})$ is a bijection for $k>\frac{3}{2}$.
\end{lemma}
\begin{proof}
    For $Y\in H^k:=H^k(\R\times\T^2,\R^{4d})$ consider the Fourier expansion
    \begin{align*}
        Y(s,t,\bar{t})&=\int\sum_{m_1,m_2\in\Z}\hat{Y}(\xi,m_1,m_2)e^{i\xi s}e^{im_1t_1}e^{im_2t_2}\,d\xi\\
        &=\int\sum_{m_1,m_2\in\Z}\hat{Y}(\xi,m_1,m_2)e^{i\xi s}e^{\frac{1}{2}t(im_1+m_2)}e^{\frac{1}{2}\bar{t}(im_1-m_2)}\,d\xi.
    \end{align*}
    Then on the Fourier coefficients the operator $D$ acts as 
    \begin{align*}
        \hat{D}(\xi,m_1,m_2) = \frac{1}{2}i\xi + \frac{1}{2}\begin{pmatrix}0&0&0&-im_1+m_2\\0&0&-im_1-m_2&0\\0&im_1-m_2&0&0\\im_1+m_2&0&0&0\end{pmatrix}-\frac{1}{2}P.
    \end{align*}
    We compute $\det(\hat{D}(\xi,m_1,m_2))=\frac{1}{16}(m_1^2+m_2^2+\xi^2+i\xi)^2$. This means only constant functions can be in the kernel of $D$, however the only constant function that is square-integrable on $\R\times\T^2$ is the zero function. So $\ker D = 0$. 

    Now, for surjectivity let $W\in H^{k-1}$ and define 
    \begin{align*}
        \hat{Y}(\xi,m_1,m_2)=\begin{cases}\hat{D}(\xi,m_1,m_2)^{-1}\hat{W}(\xi,m_1,m_2) & (\xi,m_1,m_2)\neq (0,0,0)\\
        0 & (\xi,m_1,m_2)= (0,0,0).\end{cases}
    \end{align*}
    Note that by the calculation of the determinant above $\hat{D}(\xi,m_1,m_2)^{-1}$ exists. Now by definition we have that $DY=W$ provided $Y\in H^k$. To see this, note that the eigenvalues of $\hat{D}(\xi,m_1,m_2)$ are
    \[\lambda^\pm(\xi,m_1,m_2) = \frac{1}{4}i\left(i+2\xi\pm i\sqrt{1+4m_1^2+4m_2^2}\right).\]
    For $m_1^2+m_2^2$ large enough the norms of these eigenvalues can be bounded from below. In other words, there exists some $N$ such that for $m_1^2+m_2^2>N$, we have
    \begin{align*}
       |\lambda^\pm(\xi,m_1,m_2)|^2 \geq \frac{1}{4}\xi^2+\frac{1}{8}m_1^2+\frac{1}{8}m_2^2.
    \end{align*}
    Let $\textup{pr}$ denote the projection onto the subspace where $m_1^2+m_2^2>N$. Then the $H^k$-norm of $\textup{pr}\,Y$ is
    \begin{align*}
        ||\textup{pr}\,Y||_{H^k}^2 &= \int\sum_{m_1^2+m_2^2>N}(1+\xi^2+m_1^2+m_2^2)^k|\hat{Y}(\xi,m_1,m_2)|^2\,d\xi\\
        &\leq \int\sum_{m_1^2+m_2^2>N}\frac{(1+\xi^2+m_1^2+m_2^2)^k}{\frac{1}{4}\xi^2+\frac{1}{8}m_1^2+\frac{1}{8}m_2^2}|\hat{W}(\xi,m_1,m_2)|^2\,d\xi\\
        &\leq |||W||_{H^{k-1}}^2\sup_{m_1^2+m_2^2>N}\frac{1+\xi^2+m_1^2+m_2^2}{\frac{1}{4}\xi^2+\frac{1}{8}m_1^2+\frac{1}{8}m_2^2}<+\infty.
    \end{align*}
    Here, the third line follows from H\"olders inequality. Thus $Y\in H^k$, proving that $D$ is surjective. 
\end{proof}
\begin{corollary}\label{cor:floerfredholm}
    The Floer map $\F$ is a Fredholm map of index 1.
\end{corollary}
\begin{proof}
    Note that the operator $\beta_r(s)S_Z$ is compact as a map $H^k\to H^{k-1}$. Thus, by \cref{eq:linFloer} and \cref{lem:bijection} the differential of $\F^r$ is Fredholm of index 0. Varying $r$ as well, we see that $(d\F)_{(r,Z)}$ must be Fredholm of index 1. 
\end{proof}

\subsection{Transversality}
In this section we will show that for generic choice of non-linearity the differential of the Floer map is surjective. Combined with \cref{cor:floerfredholm} this yields \cref{thm:manifold}. 

Note that the moduli space $\M$ depends on the chosen non-linearity $\beta_r(s)\tilde{F}^\rho_t$. By varying these non-linearities for each $l$, we get the universal moduli space
\[\wt{\M}:=\bigcup_{F\in\S}\M(F_{s,t}),\]
where $\S=C_c^l(\R\times \T^2\times \T^{2d}\times B,\R)$ is the space of compactly supported smooth functions $F_{s,t}$ on $\T^{2d}\times B\subseteq T^*\T^{2d}$. We denote by $\M(F_{s,t})$ the moduli space $\M$ of \cref{eq:ModuliSpace}, where the non-linearity $\beta_r(s)\tilde{F}^\rho_t$ is replaced by $F_{s,t}$. Note that in this case $F_{s,t}$ is not necessarily the product of a function in $s$, with a time-dependent non-linearity on $\T^{2d}\times B$.

This universal moduli space can be seen as the zero set of the map
\[\wt{\F}(r,Z,F)=\wt{\F}^r(Z,F)=\frac{1}{2}\del_s Z +\delslash Z - \begin{pmatrix}\del_{\bar{q}}\\\del_q\\\del_{\bar{p}}\\\del_p\end{pmatrix}H^F_{s,t}(Z),\]
where $H^F_{s,t}(Z)=\frac{1}{2}p\bar{p}+F_{s,t}(Z)$ and $F$ vanishes for $s\notin [-1,(2d+1)r+1]$. The result would follow from Sard-Smale's theorem, if we know that the differential of $\wt{\F}^r$ is surjective at every $(Z,F)\in\wt{\M}$. Let $(Z,F)\in\wt{\M}$ and take a variation $(Y,G)\in H^k(\R\times \T^2,\R^{4d})\times C^l_c(\R\times\T^2\times\T^{2d}\times B,\R)$. Then the linearization of $\wt{\F}^r$ at $(Z,F)$ can be decomposed as 
\begin{align*}
    (d\wt{\F}^r)_{(Z,F)}\cdot(Y,G) = D^1_F G +D^2_Z Y.
\end{align*}
Assume we have some non-zero $W\in H^{k-1}(\R\times \T^2,\R^{4d})$ in the orthogonal complement of the image of $(d\wt{\F}^r)_{(Z,F)}$. Then $\langle W,D^2_Z Y\rangle=0$ for all $Y$, so that $W\in \ker \left((D^2_Z)^*\right)$. Note that for $s\notin[-1,(2d+1)r+1]$ the non-linearity $F_{s,t}$ vanishes so that $(D^2_Z)^*=\frac{1}{2}\del_s-\delslash-\frac{1}{2}P$. This operator satisfies a maximum principle, similar to the one used in \cref{lem:Floerbounds}. Therefore, by standard arguments $W(s,t)\neq 0$ for some $(s,t)\in [-1,(2d+1)r+1]\times \T^2$. Thus, we can find some variation $G$ with support for $s\in[-1,(2d+1)r+1]$, such that $\langle W,D^1_F G\rangle \neq 0$. This contradicts the assumption that $\langle W,(d\wt{\F}^r)_{(Z,F)}\cdot(Y,G)\rangle=0$ for any $(Y,G)$. We must conclude $(d\wt{\F}^r)_{(Z,F)}$ is surjective. This finishes the proof of \cref{thm:manifold}.\qed

\section{Compactness}\label{sec:Compactness}
In this section we will prove compactness of the moduli space (\ref{eq:ModuliSpace}) that is needed in the proof of \cref{thm:cuplength}. We start by proving a uniform energy bound, using a standard argument.
\begin{lemma}\label{lem:Energybound}
    For all $(r,Z)\in\mathcal{M}$ we have 
    \begin{align*}
        E(Z)\leq 2||\tilde{F}^\rho||_{\textup{Hofer}}:=2\int_{\T^2}\left(\sup_{Z'\in T^*\T^{2d}}\tilde{F}^\rho_t(Z')-\inf_{Z'\in T^*\T^{2d}}\tilde{F}^\rho_t(Z')\right)\dV.
    \end{align*}
\end{lemma}
\begin{proof}
    By definition \[E(Z)=\int_{-\infty}^\infty\int_{\T^2}|\del_s Z(s,t)|^2\dV\,ds = -\int_{-\infty}^\infty\langle \grad\A_{\tilde{H}^{\rho,r}}(Z),\del_s Z\rangle\,ds,\] since $Z$ is a negative gradient line of $\A_{\tilde{H}^{\rho,r}}$. Writing this out further we get
    \begin{align*}
        E(Z) &= -\int_{-\infty}^\infty\left(d\A_{\tilde{H}^{\rho,r}}\right)_{Z_s}(\del_s Z)\,ds\\
        &=-\int_{-\infty}^\infty\frac{d}{ds}(\A_{\tilde{H}^{\rho,r}}(Z_s))\,ds+\int_{-\infty}^\infty\frac{\del\A_{\tilde{H}^{\rho,r}}}{\del s}(Z_s)\,ds\\
        &=-\lim_{s\to\infty}\A_{\tilde{H}^{\rho,r}}(Z_s)+\lim_{s\to-\infty}\A_{\tilde{H}^{\rho,r}}(Z_s)-\int_{-\infty}^\infty\int_{\T^2}\beta_r'(s)\tilde{F}^\rho_t(Z_s)\dV\,ds.
    \end{align*}
    Note that the first two terms on the last line vanish. Also, $\beta_r'$ vanishes everywhere outside $[-1,0]\cup[(2d+1)r,(2d+1)r+1]$ and is positive and negative respectively on those two intervals. Therefore,
    \begin{align*}
        E(Z) &=-\int_{-1}^0\int_{\T^2}\beta_r'(s)\tilde{F}^\rho_t(Z_s)\dV\,ds - \int_{(2d+1)r}^{(2d+1)r+1}\int_{\T^2}\beta_r'(s)\tilde{F}^\rho_t(Z_s)\dV\,ds\\
        &\leq 2\int_{\T^2}\left(\sup_{Z'\in T^*\T^{2d}}\tilde{F}^\rho_t(Z')-\inf_{Z'\in T^*\T^{2d}}\tilde{F}^\rho_t(Z')\right)\dV\\
        &=:2||\tilde{F}^\rho||_{\textup{Hofer}}.
    \end{align*}
\end{proof}

Now we want to relate the energy to the $L^2$-norm of the derivatives of $Z$. To this end we define the energy density
\[e_Z(s,t) = \frac{1}{2}|dZ(s,t)|^2 = \frac{1}{2}|\del_s Z(s,t)|^2+|\del_t Z(s,t)|^2+|\del_{\bar{t}} Z(s,t)|^2.\]
\begin{lemma}\label{lem:evsE}
    Let $K=[-M,M]\times\T^{2}\subseteq \R\times\T^{2}$. Then there exists a constant $C=C(K)$ such that for any $(r,Z)\in\M$ 
    \[\int_{K}e_Z(s,t)\dV\,ds \leq E(Z) + C(K).\]
\end{lemma}
\begin{proof}
    We calculate\footnote{The first equality follows from partial integration, using that $Z$ is nullhomotopic.}
    \begin{align*}
        \int_{K}e_Z(s,t)\dV\,ds &=\frac{1}{2}\int_K|\del_s Z|^2 \dV\,ds +2\int_K \left(|\del_t q|^2+|\del_t p|^2\right)\dV\,ds\\
        &=\frac{1}{2}\int_K|\del_s Z|^2 \dV\,ds +2\int_K \left(\left|\frac{\del H}{\del p}-\frac{1}{2}\del_s\bar{p}\right|^2+\left|\frac{1}{2}\del_s\bar{q}-\frac{\del H}{\del q}\right|^2\right)\dV\,ds\\
        &\leq E(Z) + 2\int_K\left(\left|\frac{\del H}{\del q}\right|^2+\left|\frac{\del H}{\del p}\right|^2\right)\dV\,ds\\
        &\leq E(Z) + \frac{1}{2}\int_K|p|^2\dV\,ds +2\int_K \left(\left|\beta_r(s)\frac{\del \tilde{F}^\rho}{\del q}\right|^2+\left|\beta_r(s)\frac{\del\tilde{F}^\rho}{\del p}\right|^2\right)\dV\,ds.
    \end{align*}
    Now, as in \cref{lem:Floerbounds} we have that $|p(s,t)|^2<\rho$ for all $(s,t)\in K$. Combined with the $C^1$-boundedness of $F$, we can estimate the right-hand side purely in terms of $E(Z)$ and the volume of $K$. This proves the lemma. 
\end{proof}

In order to get point-wise bounds on the derivatives, we want to use the Heinz trick (\cref{thm:Heinz}). We first need to show that the energy density satisfies an appropriate inequality.
\begin{lemma}\label{lem:densityineq}
    The energy density $e:=e_Z$ satisfies the inequality
    \begin{align}\label{eq:densityineq}
        (\del_s^2+\Delta)e\geq -c(1+e^{\frac{3}{2}}),
    \end{align}
    for some $c>0$.
\end{lemma}
\begin{proof}
    See \cref{appendix}.
\end{proof}
\begin{theorem}[Heinz trick, see \cite{hohloch2009hypercontact}]\label{thm:Heinz}
    Let $M$ be an $(m+1)$-dimensional Riemannian manifold for $\frac{3}{2}<\frac{m+3}{m+1}$ and $e:M\to\R_{\geq0}$ a smooth function satisfying \cref{eq:densityineq}. Then for $K\subseteq M$ compact $\sup_K e$ is bounded by $\int_K e$.
\end{theorem}

Now we are ready to prove the desired compactness theorem.
\begin{theorem}\label{thm:compactness}
    The moduli space $\M$ is relatively compact in the $C^\infty_{\loc}$-topology.
\end{theorem}
\begin{proof}
    Let $(r_n,Z_n)\in\M$ for all $n\in\N$ and $K=[-M,M]\times\T^{2}\subseteq\R\times\T^2$. By \cref{lem:Floerbounds} the sequence $Z_n\vert_K$ is a bounded family of smooth functions on $K$. By combining \cref{lem:Energybound,lem:evsE} we find a uniform bound on $\int_K e_{Z_n}$. Then the Heinz trick (\cref{thm:Heinz}) combined with the density estimate in \cref{lem:densityineq} yields a uniform pointwise bound on $e_{Z_n}(s,t)$ for all $(s,t)\in K$. By the mean-value theorem it follows that the sequence $(Z_n)$ is uniformly equicontinuous. The theorem of Arzelà-Ascoli then yields a converging subsequence in the $C^0_{\loc}$-topology. Finally, the ellipticity of the Floer operator\footnote{Note that $(\frac{1}{2}\del_s-\delslash)(\frac{1}{2}\del_s+\delslash)=\frac{1}{4}\del_s^2+\del_t \del_{\bar{t}}$, so the Floer operator is a factor of the Laplacian on $\R\times\T^2$.} yields convergence of the subsequence in $C^\infty_{\loc}$.
\end{proof}

\appendix
\section{Proof of \cref{lem:densityineq}}\label{appendix}
The proof of \cref{lem:densityineq} is easier to write out in the 'real' notation using $T=(t_1,t_2)\in\T^2$ and $Z=(q_1,q_2,p_1,p_2)$. Note that it is just a matter of notation so that the proof still holds in the complex notation as well. Recall that $J_\del$ is an operator satisfying $J_\del^2=-(\del_1^2+\del_2^2)$. We can write $J_\del=J_1\del_1+J_2\del_2$ for some matrices $J_1,J_2$ satisfying $J_1^2=J_2^2=-\Id$ and $J_1J_2+J_2J_1=0$ (see \cite{bridgesTEA}). For any $(r,Z)$ in $\M$, it holds that $Z$ satisfies the Floer equation
\begin{align}\label{eq:FloerJdel}
    \del_s Z +J_\del Z = \nabla \tilde{H}^{\rho,r}(Z),
\end{align}
where $\tilde{H}^{\rho,r}_{s,T}(Z)=\frac{1}{2}(p_1^2+p_2^2)+\beta_r(s)\tilde{F}^\rho_{T}(Z)$. We define $t_0:=s$ and $\xi_i:=\del_iZ$ for $i=0,1,2$. From the identity 
\[(\del_s-J_\del)(\del_s+J_\del)=\L:=\sum_{i=0}^2\del_i^2\]
we see that the $\xi_i$ satisfy the equation
\begin{align}\label{eq:xiLaplace}
    \L\xi_i = (\del_s-J_\del)(\sigma_Z\xi_i),
\end{align}
where $\sigma_Z$ is the Hessian of $\tilde{H}^{\rho,r}$ at $Z$.

Note that in this notation the energy density is given by 
\[e(s,T) = \frac{1}{2}\sum_{i=0}^2|\xi_i(s,T)|^2.\]
We calculate
    \begin{align*}
        \L e &= \sum_{i,j=0}^2|\del_j\xi_i|^2 + \sum_{i=0}^2\langle \xi_i,\L \xi_i\rangle\\
        &=\sum_{i,j=0}^2|\del_j\xi_i|^2 + \sum_{i=0}^2\langle \xi_i,(\del_s-J_\del)(\sigma_Z\xi_i)\rangle,
    \end{align*}
where we have used \cref{eq:xiLaplace} in the second equality. Now note that the chain rule gives that
    \begin{align*}
        \del_j(\sigma_Z \xi_i)=T\sigma_Z\cdot \xi_j\cdot \xi_i+\sigma_Z\del_j \xi_i.
    \end{align*}
Denote $J_0=-\Id$. Then
    \begin{align*}
        \L e&= \sum_{i,j=0}^2\left(|\del_j\xi_i|^2 -\langle (J_j \sigma_Z)^T\xi_i,\del_j\xi_i\rangle-\langle\xi_i,J_jT\sigma_Z\cdot\xi_j\cdot\xi_i\rangle\right)\\
        &=\sum_{i,j=0}^2\left(\left|\del_j\xi_i-\frac{1}{2}(J_j\sigma_Z)^T\xi_i\right|^2-\frac{1}{4}\left|(J_j\sigma_Z)^T\xi_i\right|^2-\langle\xi_i,J_jT\sigma_Z\cdot\xi_j\cdot\xi_i\rangle\right).
    \end{align*}
Note that since $F$ is $C^3$-bounded, we have that both $\sigma_Z$ and $T\sigma_Z$ are bounded. Thus
\[
    \left|(J_j\sigma_Z)^T\xi_i\right|^2\leq c_1|\xi_i|^2\leq c_2 e
\]
and
\[
    \left|\langle\xi_i,J_jT\sigma_Z\cdot\xi_j\cdot\xi_i\rangle\right|\leq c_3|\xi_i|^2|\xi_j|\leq c_4 e^{\frac{3}{2}}.
\]
Combined we get that
\[
\L e\geq -c_5\left(e+e^{\frac{3}{2}}\right).
\]
Finally, Young's inequality with $p=3$, $q=\frac{3}{2}$ gives that
    \[e\leq \frac{1}{3}+\frac{2}{3}e^{\frac{3}{2}},\]
    so that
    \[\L e\geq -c(1+e^{\frac{3}{2}}).\]
\qed

\bibliography{mybib}{}
\bibliographystyle{alpha}

\end{document}